\documentclass[12pt]{amsart}
\usepackage{amssymb, amsmath, latexsym, graphics, graphicx}

\newcommand{\GL}{{\rm GL}}
\newcommand{\SL}{{\rm SL}}
\newtheorem{theorem}{Theorem}[section]
\newtheorem*{thm1}{Theorem A}
\newtheorem*{thm2}{Theorem B}
\newtheorem*{thm3}{Theorem C}
\newtheorem{lemma}[theorem]{Lemma}
\newtheorem{proposition}[theorem]{Proposition}
\newtheorem{corollary}[theorem]{Corollary}
\newtheorem{remark}[theorem]{Remark}
\begin{document}
\date{May 2011}
\thanks{Supported by EPSRC grant EP/G025487/1}
\title{Lie powers of the natural module for $\GL(2,K)$}
\author{Karin Erdmann}
\address[K. Erdmann]{Mathematical Institute, 24--29 St Giles', Oxford, OX1 3LB, UK}
\email{erdmann@maths.ox.ac.uk}

\author{Marianne Johnson}
\address[M. Johnson]{School of Mathematics, University of Manchester,
Alan Turing Building, Manchester M13 9PL, UK}
\email{marianne.johnson@maths.man.ac.uk}
\numberwithin{equation}{section}
\baselineskip=16pt
\begin{abstract}
In recent work of R. M. Bryant and the second author a (partial) modular analogue of Klyachko's 1974 result on Lie powers of the natural $\rm{GL}(n,K)$ was presented. There is was shown that nearly all of the indecomposable summands of the $r$th tensor power also occur up to isomorphism as summands of the $r$th Lie power provided that $r\neq p^m$ and $r \neq 2p^m$, where $p$ is the characteristic of $K$. In the current paper we restrict attention to ${\rm GL}(2,K)$ and consider the missing cases where $r = p^m$ and $r = 2p^m$. In particular, we prove that the indecomposable summand of the $r$th tensor power of the natural module with highest weight $(r-1,1)$ is a summand of the $r$th Lie power if and only if $r$ is a not power of $p$.
\end{abstract}
\maketitle
\section{Introduction}
Let $E$ denote the natural module for the general linear group $G=\GL(n,K)$ over an infinite field $K$ of characteristic $p \geq 0$. The isomorphism types of the indecomposable summands of $E^{\otimes r}$ are parameterized by (row) $p$-regular partitions of $r$ into at most $n$ parts. We denote these summands by $T(\lambda)$. For each $p$-regular partition $\lambda$ of $r$ let $D^{\lambda}$ denote the simple $KS_r$-module labelled by $\lambda$. Then
\begin{equation}\label{tensortilt}
E^{\otimes r} \cong \bigoplus d_\lambda T(\lambda),
\end{equation}
where $d_\lambda = \dim D^{\lambda}>0$ and the sum ranges over all $p$-regular partitions of $r$ into at most $n$ parts. (Here we use the notation $nV$ to denote $V \oplus \cdots \oplus V$.)
Now let $L^r(E)$ denote the $r$th homogeneous component of the free Lie algebra $L(E)$. We take $L(E)$ to be the Lie subalgebra of the tensor algebra of $E$ generated by $E$;  then $L^r(E)$ is a $KG$-submodule of $E^{\otimes r}$. Moreover, if $r$ is not divisible by $p$ then it is well known that $L^r(E)$ is a direct summand of $E^{\otimes r}$. Thus, for $p \nmid r$ we have
\begin{equation}\label{Lietilts}
L^r(E) \cong \bigoplus  l_\lambda T(\lambda),
\end{equation}
where the sum ranges over all $p$-regular partitions of $r$ into at most $n$ parts and $0 \leq l_\lambda \leq d_\lambda$. Donkin and Erdmann~\cite{DE} gave a formula describing the multiplicities $l_\lambda$ in terms of Brauer characters of the symmetric group $S_r$, as follows:
\begin{equation*}
l_\lambda = \frac{1}{r}\sum_{d \mid r} \mu(d) \beta^\lambda(\sigma^{r/d}),
\end{equation*}
where $\mu$ is the M{\"o}bius function, $\beta^\lambda$ is the Brauer character of $D^\lambda$ and $\sigma = (1 2 \cdots r) \in S_r$. One would like to be able to calculate the multiplicities $l_\lambda$, however, the Brauer characters are not known in general. In particular, it is difficult to determine from this formula alone which multiplicities are non-zero.

In characteristic zero, Klyachko~\cite{K} has shown that almost all of the irreducible $KG$-submodules of the $r$th tensor power $E^{\otimes r}$ also occur up to isomorphism as submodules of the $r$th Lie power $L^r(E)$. Since $E^{\otimes r}$ is completely reducible in this case we obtain that the multiplicities $l_\lambda$ occurring on the right-hand side of \eqref{Lietilts} are almost always positive. In the spirit of this result we would like to know, for arbitrary characteristic, which indecomposable summands of the $r$th tensor power also occur up to isomorphism as summands of $L^r(E)$. For modules $U$ and $V$ we write $U \mid V$ to mean that $U$ is isomorphic to a direct summand of $V$. Thus, by \eqref{tensortilt}, we would like to know for which $p$-regular partitions $\lambda$ of $r$ into at most $n$ parts we have $T(\lambda) \mid L^r(E)$.

When $K$ is an infinite field of prime characteristic $p$ Klyachko's original argument can be modified to prove a similar result for Lie powers of certain degree, see~\cite{BJ}. Unfortunately the methods used there do not work well when the degree is a power of $p$ or twice a power of $p$. Throughout this paper we shall restrict attention to the case where $K$ is an infinite field of prime characteristic $p$ and $G=\GL(2,K)$. We shall prove the following theorems.

\begin{thm1}
Let $K$ be an infinite field of characteristic $2$, $G = \GL(2,K)$ and let $E$ denote the natural $KG$-module. Let $r$ be a positive integer greater than $6$ and $\lambda$ a $2$-regular partition of $r$ into at most two parts.
\begin{itemize}
\item[(i)] If $r$ is not a power of $2$ then $T(\lambda) \mid L^r(E)$ if and only if $\lambda \neq (r)$.
\item[(ii)] If $r$ is a power of $2$ then $T(\lambda) \mid L^r(E)$ if and only if  $\lambda \neq (r), (r-1, 1)$.
\item[(iii)] Let $1 \leq t_1 <  t_2 <  \ldots < t_k$ be such that
 $r=2s_i+3t_i$ with $s_i \geq 1$. Then $\bigoplus_{i=1}^k E^{\otimes t_i} \mid L^r(E)$, considered as modules for $\SL(2,K)$.
\end{itemize}
\end{thm1}

Part (i) of the theorem is a special case of~\cite[Theorem 6.8]{BJ}. We give an alternative proof here, using a result of St\"{o}hr~\cite[Corollary 9.2]{S} on free Lie algebras of rank two in characteristic $2$. Part (ii) of the theorem deals with the cases not covered by~\cite[Theorem 6.8]{BJ} when $p=n=2$. Note that part (iii) ca be used to give a fairly large lower bound for the multiplicity of a given indecomposable tilting module $T(\lambda)$ as a direct summand of $L^r(E)$. The precise statement is given in Corollary \ref{bound-cor}. In Section 3 we lay the groundwork for the proof of Theorem A by exploiting~\cite[Corollary 9.2]{S}. The remainder of the proof then follows from the following theorem (to be proved in Section 4) in the case where $p=2$.

\begin{thm2}
Let $K$ be an infinite field of prime characteristic $p$, $G = \GL(2,K)$ and let $E$ denote the natural $KG$ module. Then $T(r-1,1)$ is a summand of $L^r(E)$ if and only if either $r=p$ or $r$ is not a power of $p$.
\end{thm2}

Furthermore, in odd characteristic we prove the following:

\begin{thm3} Let $K$ be an infinite field of odd characteristic $p$, $G=\GL(2,K)$ and
let $E$ be the natural $KG$ module. Let $r>p$ and let $\lambda$ be a partition of $r$ into at most two parts.\\
(i) If  $r=p^m$ with $p>3$ then
then $T(\lambda) \mid L^r(E)$ if and only if $\lambda \neq (r), (r-1,1)$.\\
(ii) Let $r=p^m$ with $p=3$ and suppose $\lambda \neq (r), (r-1, 1),
((r+1)/2, (r-1)/2)$. Then $T(\lambda) \mid L^r(E)$.\\
(iii) Let $r=2p^m$ with $p>3$ and suppose $\lambda  \neq (r), (p^m, p^m)$. Then $T(\lambda) \mid L^r(E)$.\\
(iv) Let $r=2p^m$ with $p=3$ and suppose $\lambda \neq (r), (p^m, p^m), (p^m+1, p^m-1), (p^m+2, p^m-2)$. Then $T(\lambda) \mid L^r(E)$.
\end{thm3}
This is proved by using a result of St\"{o}hr and Vaughan-Lee~\cite[Theorem 1]{VLS}.

\section{Polynomial representations of $\GL(2,K)$}
\label{poly}
Let $K$ be an infinite field of prime characteristic $p$, $n$ a fixed positive integer, and let $G=\GL(n,K)$. We begin by recalling a few facts about polynomial $KG$-modules (see~\cite{G} for further reference) and then quickly specialize to the case $n=2$.

\medskip
{\bf 2.1 } \  Let $E$ denote the natural $KG$-module. Then $E$ is a polynomial module of degree 1. If $V$ is a polynomial module of degree $r$ and $W$ is a polynomial module of degree $s$ then $V \otimes W$ is a polynomial module (with diagonal action) of degree $r+s$. Every submodule and every quotient of a polynomial module of degree $r$ is polynomial of degree $r$. Thus we see that $E^{\otimes r}$ and $L^r(E)$ are polynomial modules of degree $r$.

Let $\Lambda(n,r)$ denote the set of unordered partitions of $r$ into at most $n$ parts. Any polynomial module $V$ of degree $r$ can be written as a direct sum of weight spaces over $K$,
$$
V = \bigoplus_{\alpha \in \Lambda(r)}  V^\alpha,
$$
where $V^\alpha$ is the $K$-vector space
of all $v$ such that
$tv = t_1^{\alpha_1}t_2^{\alpha_2}\ldots t_n^{\alpha_n}v$
with
$t\in G$ diagonal with $i$th diagonal entry $t_i$.
Let $\Lambda^+(n, r)$ denote the set of (ordered) partitions of $r$ into at most $n$ parts. The simple polynomial modules of degree $r$ are indexed by $\Lambda^+(n, r)$ and we denote the simple module labelled by the partition $\lambda$ by $L(\lambda)$. To every partition $\lambda \in \Lambda^+(n,r)$ let $\Delta(\lambda)$ denote the \emph{Weyl module} with unique simple quotient $L(\lambda)$ (see~\cite[(5.3a),(5.3b) and (5.4b)]{G}). All other composition factors of $\Delta(\lambda)$ are isomorphic to modules of the form $L(\mu)$ where $\mu < \lambda$ with respect to the dominance ordering. Let $\nabla(\lambda)$ denote the contravariant dual of $\Delta(\lambda)$. We say that a finite-dimensional $KG$-module admits a \emph{Weyl filtration} if it has a filtration in which every section is isomorphic to a Weyl module. Similarly we say that a finite-dimensional $KG$-module admits a \emph{dual Weyl filtration} if it has a filtration in which every section is isomorphic to a dual Weyl module.

The category of polynomial $KG$-modules of degree $r$ is equivalent to the module category of the Schur algebra $S(n,r)$ (for details see~\cite{G}). This has the advantage that there are finite-dimensional
injective modules. This category is also a highest weight category in the sense of~\cite{CPS}, with weight poset given by $\Lambda^+(n,r)$ with respect to the dominance ordering. The Weyl modules are the `standard modules' and the dual Weyl modules the `costandard modules'. Equivalently, the Schur algebra $S(n,r)$ is quasi-hereditary.
By a theorem of Ringel~\cite{R}, and also Donkin~\cite{D} for the algebraic group setup,  each highest weight category
has a class of `canonical' modules $T(\lambda)$ indexed by dominant weights.
These are the indecomposable polynomial modules of degree $r$ admitting both a Weyl filtration and a dual Weyl filtration. We shall call the modules $T(\lambda)$ the \emph{indecomposable tilting modules} of degree $r$.
Every Weyl filtration of $T(\lambda)$ contains exactly one section isomorphic to $\Delta(\lambda)$; all other sections are isomorphic to $\Delta(\mu)$, where $\mu < \lambda$. In fact, we have something stronger:

\begin{remark}\label{weight} \normalfont \ The $\lambda$-weight space of $T(\lambda)$ is one-dimensional, and the
submodule generated by a weight vector of weight $\lambda$ is isomorphic
to $\Delta(\lambda)$. This is the unique submodule of $T(\lambda)$ isomorphic
to $\Delta(\lambda)$.
\end{remark}

We shall say that a polynomial module is a \emph{tilting module} if it is isomorphic to a direct sum of indecomposable tilting modules. An important property is that the tilting modules are closed under tensor products; this
follows from the fact that tensor products of Weyl modules
(respectively dual Weyl modules) have Weyl filtration (respectively dual Weyl filtration). Thus if $V$ and $W$ are tilting modules of degree $r$ and $s$ then $V \otimes W$ is a tilting module of degree $r+s$.

\medskip
{\bf 2.2 } \ \label{n=2} From now on we assume $n=2$ so that $G=\GL(2,K)$. We fix the  basis $\{x,y\}$ of $E$ such that, when $E$ is identified with the column space, the vectors $x, y$ are identified with ${1\choose 0}$ and ${0\choose 1}$ respectively.

To simplify notation we shall write $\Lambda^+(r)$ to mean $\Lambda^+(2,r)$. Recall that a partition is (row) $p$-regular if no $p$ parts are equal. We write $\Lambda^+_p(r)$ for the set of  $p$-regular partitions in $\Lambda^+(r)$. Thus for $p>2$ we have $\Lambda^+_p(r)=\Lambda^+(r)$, whilst $\Lambda^+_2(r)$ consists of all partitions $\lambda = (\lambda_1, \lambda_2)$ of $r$ with $\lambda_1 >\lambda_2$.
The isomorphism types of the indecomposable summands of $E^{\otimes r}$ are given by the $T(\lambda)$ where $\lambda \in \Lambda^+_p(r)$, as described in \eqref{tensortilt}.
Thus we would like to know for which $\lambda \in \Lambda^+_p(r)$ we have $T(\lambda) \mid L^r(E)$.

It will often be convenient to work with the subgroup $H=\SL(2,K)$. Let $\mathcal{P}_r$ denote
the category of $KH$-modules which are restrictions of polynomial $KG$-modules of degree $r$. Then $\mathcal{P}_r$ is also a highest weight category. We identify the set of dominant weights with the set $$W_r=\{m \in \mathbb{N}_0: 0 \leq m \leq r, m \equiv r \bmod2\}.$$
The simple  modules in $\mathcal{P}_r$ are then indexed by $W_r$ and we denote these by $L(m)$.
Similarly, we denote the Weyl modules, dual Weyl modules and indecomposable tilting modules in $\mathcal{P}_r$ by $\Delta(m)$, $\nabla(m)$ and $T(m)$ respectively. We say that a module in $\mathcal{P}_r$ is a tilting module if it is isomorphic to a direct sum of indecomposable tilting modules.

If $\lambda=(\lambda_1, \lambda_2) \in \Lambda^+(r)$ then the simple module $L(\lambda)$ of $\GL(2,K)$ restricts to  $L(m)$, where $m=\lambda_1 - \lambda_2$. Similarly, $\Delta(\lambda), \nabla(\lambda)$ and $T(\lambda)$ restrict to $\Delta(m), \nabla(m)$ and $T(m)$, respectively.
Suppose $M$ is any $S(2,r)$-module whose restriction to
$\SL(2,K)$ is isomorphic to $L(m)$, or $\Delta(m)$, or $\nabla(m)$, or $T(m)$.
Then $m\leq r$ and $m\equiv r \bmod 2$. Moreover
there is a unique partition $\lambda(m) \in \Lambda^+(r)$
 such that $M$ is isomorphic to $L(\lambda(m))$, or $\Delta(\lambda(m))$, or
$\nabla(\lambda(m))$ or $T(\lambda(m))$ (see
~\cite[3.2.7]{D-LN}). We note that the restriction of \eqref{tensortilt} to $\SL(2,K)$ is given by $E^{\otimes r} \cong \bigoplus d_\lambda T(m)$,
where the sum ranges over all $m$ in the set
\begin{equation}
\label{A}
A_r = \{ k \in \mathbb{N}: 0<k\leq r \mbox{ and } k \equiv r \bmod 2\}.
\end{equation}

We summarize some properties of the modules $T(m)$ and $\Delta(m)$ (see~\cite{D}, ~\cite{X} and ~\cite{W} for further details).
For a $KG$- or $KH$-module $M$, we denote its Frobenius twist by $M^F$.

\begin{itemize}
\item[(2.2a)] For any $m\geq 0$, $\nabla(m)$ is isomorphic to the
$m$-th symmetric power giving $$\dim \Delta(m)=\dim \nabla(m)=m+1.$$
\item[(2.2b)] Let $m$ be a non-negative integer. Then $$T(m) \cong \Delta(m) \cong \nabla(m) \cong L(m)$$ if and only if either $m=0$ or $m=ap^k -1$, where $2\leq a \leq p$ and $k \geq 0$ (see ~\cite{W}). In particular,  $E=\Delta(1) \cong \nabla(1) \cong L(1) \cong T(1)$ and $\Delta(0)=K$ as $SL(2,K)$-modules.
\item[(2.2c)] If $m=kp+i$ where $0\leq i\leq p-2$ and $k \geq 1$ then $$T(m) \cong T(k-1)^F \otimes T(p+i).$$
\item[(2.2d)] If $m=kp+(p-1)$ where $k \geq 0$ then $T(m) \cong T(k)^F \otimes T(p-1)$. If $T(k)$ is simple then so is $T(m)$.
\item[(2.2e)] For all $i,j \geq 0$ satisfying $i+j=p-2$ there is a short exact sequence of $\SL(2,K)$-modules given by
$$
0 \rightarrow \Delta(p+j) \rightarrow T(p+j) \rightarrow \Delta(i) \rightarrow 0.
$$
\item[(2.2f)] For all $m,n\geq0$ with $m\geq n$ the tensor product $\Delta(m) \otimes \Delta(n)$ has a Weyl filtration with sections $$\Delta(m+n), \Delta(m+n-2), \ldots, \Delta(m-n).$$
\item[(2.2g)] For all $m,n\geq0$, $T(n+m)$ is a direct summand of $T(n) \otimes T(m)$.
\item[(2.2h)] If $0\leq i, j\leq p-2$ and $i+j=p-2$ then for
any $n\geq 1$ there is a short exact sequence
$$0\to \Delta(n-1)^F\otimes L(j) \to \Delta(pn+i) \to \Delta(n)^F\otimes L(i)
\to 0.
$$
\end{itemize}

\section{Lie powers of the natural module in characteristic $2$}
In this section we lay the groundwork for the proof of Theorem A. The following result of St\"{o}hr on free Lie algebras of rank two in characteristic two will be a key ingredient of our proof.

\begin{theorem}~\cite[Corollary 9.2]{S}
\label{Ralph}
Let $K$ be a field of characteristic $2$, $G$ a group and let $V$ be a two-dimensional $KG$-module.
For all $r\geq4$ there is a direct sum decomposition of $L^r(V)$ as a $KG$-module:
$$
L^r(V)= \bigoplus_{s+t \geq 1} m_{s,t} L^{r/(2s+3t)}(R^2(V)^{\otimes s} \otimes R^3(V)^{\otimes t}),
$$
where $L^{r/d}(X)=0$ if $d \nmid r$ and $m_{s,t} = \frac{1}{s+t} \sum_{d \mid s,t} \mu(d) \frac{\left((s+t)/d\right)!}{\left(s/d\right)!\left(t/d\right)!}$.
\end{theorem}

Here $R(V)$ is the free restricted Lie algebra on $V$.
Notice that the multiplicity $m_{s,t}$ is equal to the dimension of the subspace of $L^{s+t}(V)$ spanned by the monomials of multidegree $(s,t)$, given by Witt's dimension formula (see~\cite[Theorem 5.11]{M} for example). In particular, these multiplicities are all positive.

We apply Theorem \ref{Ralph} in the case where $K$ is an infinite field of characteristic $2$, $G=\GL(2,K)$ and $V=E$ is the natural $KG$-module. Let $\{x,y\}$ be the basis of $E$ as defined in Section \ref{n=2}. We claim that the modules $R^2(E)$ and $R^3(E)$ occurring in this decomposition can be identified with certain Weyl modules.
By definition $\Delta(2,0)$ is the submodule of $E^{\otimes 2}$ generated by $x \otimes x$ (see~\cite[(5.3b)]{G}). It has basis $\{x \otimes x, y \otimes y, x \otimes y + y\otimes x\}$. Since the characteristic of $K$ is two it then follows that $R^2(E) \cong \Delta(2,0)$.
Similarly, $\Delta(2,1)$ is the submodule of $E^{\otimes 3}$ generated by $x \otimes x \otimes y + y \otimes x \otimes x$. It has basis $\{x \otimes x \otimes y + y \otimes x \otimes x, y \otimes y \otimes x + x \otimes y \otimes y\}$. Since $K$ has characteristic $2$ we have $R^3(E) = L^3(E)$ with basis given by $\{[x,y,x], [x,y,y]\}$. Thus it is easy to see that $R^3(E) \cong \Delta(2,1)$.

Let $D_{s,t} = \Delta(2,0)^{\otimes s} \otimes \Delta(2,1)^{\otimes t}$. Then Theorem \ref{Ralph} yields
\begin{equation}
\label{Ralpheq}
L^r(E) \cong L^{r/2}(D_{1,0}) \oplus L^{r/3}(D_{0,1})  \oplus \bigoplus_{s,t \geq 1} m_{s,t} L^{r/(2s+3t)}(D_{s,t}),
\end{equation}
for all $r\geq 4$.
Since the multiplicities occurring on the right-hand side of \eqref{Ralpheq} are all positive we see that $D_{s,t}$ is isomorphic to a direct summand of $L^{2s+3t}(E)$ for all $s,t \geq 1$. We shall show that each such summand $D_{s,t}$ is a tilting module for $G$. Let $\Delta_{s,t}$ denote the restriction of $D_{s,t}$ to $\SL(2,K)$. Thus $\Delta_{s,t} = \Delta(2)^{\otimes s} \otimes \Delta(1)^{\otimes t}$. We shall soon see that it is enough to show that $\Delta_{s,t}$ is a tilting module for $\SL(2,K)$.

\begin{lemma}
\label{deltatilt}
Let $K$ be an infinite field of characteristic $2$, let $H=\SL(2,K)$ and let $E$ denote the
natural $KH$-module. For each $s,t \geq1$ let $\Delta_{s,t}$ denote the $KH$-module defined by $\Delta_{s,t}= \Delta(2)^{\otimes s} \otimes \Delta(1)^{\otimes t}$. Then $\Delta_{s,t} \mid E^{\otimes 2s+t}$.
\end{lemma}
\begin{proof}
We first consider the case where $s=t=1$.  By \eqref{tensortilt} with $r=p=2$ we obtain
\begin{equation}
\label{esquare}
E^{\otimes 2} \cong T(2).
\end{equation}
Thus, by (2.2e), there is short exact sequence
$$0\to \Delta(2) \to E^{\otimes 2} \to \Delta(0)\to 0$$
of $KH$-modules. Tensoring this with $E \cong \Delta(1)$
gives
\begin{equation}
\label{ecubeseq}
0\to  \Delta(2)\otimes \Delta(1) \to E^{\otimes 3} \to \Delta(1)\to 0.
\end{equation}
Now, by \eqref{tensortilt} with $r=3$ and $p=2$, we have
\begin{equation}
\label{ecube}
E^{\otimes 3} \cong T(3) \oplus 2T(1),
\end{equation}
so that,  by (2.2b) and (2.2d),  the middle term of \eqref{ecubeseq} is semisimple. Thus the sequence \eqref{ecubeseq} is split and we deduce that
\begin{equation}
\label{delta11}
\Delta_{1,1} =\Delta(2) \otimes \Delta(1) \cong T(3) \oplus T(1).
\end{equation}
In particular,  $\Delta_{1,1}$ is isomorphic to a direct summand of $E^{\otimes 3}$ and hence $\Delta_{1,t}=\Delta(2)\otimes E^{\otimes t}=\Delta_{1,1} \otimes E^{\otimes t-1}$ is isomorphic to a direct summand of $E^{\otimes t+2}$ for all $t \geq 1$. It follows that $\Delta(2) \otimes T \mid E^{\otimes t+2}$ for all $T \mid E^{\otimes t}$ and hence by induction that $\Delta_{s,t} \mid E^{\otimes 2s+t}$ for all $s,t \geq 1$.\end{proof}

\begin{corollary}
\label{Dtilt}
Let $K$ be an infinite field of characteristic $2$, let $G=\GL(2,K)$ and let $E$ denote the
natural $KG$-module. For each $s,t \geq1$ let $D_{s,t}$ denote the $KG$-module defined by $D_{s,t}= \Delta(2,0)^{\otimes s} \otimes \Delta(2,1)^{\otimes t}$. Then $D_{s,t} \mid E^{\otimes 2s+3t}$.
\end{corollary}

\begin{proof}
Let $H=\SL(2,K)$ and for all positive integers $m$ and $a$ let $g(m,a)$ denote the multiplicity of $T(m)$ in $E^{\otimes a}$ considered as a $KH$-module. By Lemma \ref{deltatilt} we have that the restriction $\Delta_{s,t}$ of $D_{s,t}$ to $H$ is a summand of $E^{\otimes 2s+t}$. Thus we may write
$$\Delta_{s,t} = \bigoplus_{m\in A_{2s+t}} a_mT(m)$$
for some multiplicities $a_m$ satisfying $0 \leq a_m \leq g(m, 2s+t)$. Note that the sum ranges over all $m \in A_{2s+t}$ with $A_r$ as defined in \eqref{A}. Since $A_{2s+t}$ is a subset of $A_{2s+3t}$, each indecomposable module $T(m)$ occurring on the right-hand side is a $KH$-summand of $E^{\otimes 2s+3t}$. We shall show that each such summand $T(m)$ occurs in $E^{\otimes 2s+3t}$ with multiplicity greater than or equal to $a_m$. Then by the unique lifting of $T(m)$ to a tilting module for $G$ (see Section \ref{n=2}) it will follow that $D_{s,t}$ is isomorphic to a $KG$-summand of $E^{\otimes 2s+3t}$.

Thus we must show that $a_m \leq g(m, 2s+3t)$ for all $m \in A_{2s+t}$. Since $a_m \leq g(m, 2s+t)$, it is enough to show that $g(m, 2s+t) \leq g(m, 2s+3t)$. Now, by~\cite[Lemma 1.7.2 and Lemma 1.5(1)]{E2}, it is known that for any positive integers $a$ and $m$ we have $g(m,a) \leq g(m, a+2)$. This completes the proof.
\end{proof}

We shall now show that every indecomposable summand of $E^{\otimes 2s+t}$ is isomorphic to a direct summand of $\Delta_{s,t}$ as $\SL(2,K)$-modules.

\begin{lemma}
\label{Deltaparts}
Let $K$ be an infinite field of characteristic $2$, let $H=\SL(2,K)$ and let $E$ denote the
natural $KH$-module. Let $s,t\geq 1$. Then $T(m) \mid \Delta_{s,t}$ if and only if $ T(m) \mid E^{\otimes 2s+t}$.
\end{lemma}
\begin{proof}
Let $r=2s+t$ and $p=2$. As we have seen, the restriction of \eqref{tensortilt} to $H$ yields that $ T(m) \mid E^{\otimes 2s+t}$ if and only if $m \in A_{2s+t}$, where $A_{2s+t}$ is as in \eqref{A}. Thus we must show that $T(m) \mid \Delta_{s,t}$ if and only if $m \in A_{2s+t}$. By Lemma \ref{deltatilt} we may write
\begin{equation}
\label{star}
\Delta_{s,t} \cong \bigoplus_{m \in A_{2s+t}} a_m T(m)
\end{equation}
for some multiplicities $a_m \geq 0$. So it is enough to show that each of the multiplicities $a_m$ occurring on the right-hand side of \eqref{star} is non-zero.

We first note that this holds for $\Delta_{1,1}$. Indeed $A_{2(1)+1} = \{1,3\}$ and by \eqref{delta11} $\Delta_{1,1} \cong T(3) \oplus T(1)$.  Thus we may suppose that the multiplicities $a_m$ occurring on the right-hand side of \eqref{star} are all positive for some $s,t \geq1$ and proceed by induction on $s$ and $t$. Since $\Delta(1) = T(1)$, from \eqref{star} we obtain
$$\Delta_{s,t+1}=\Delta_{s,t} \otimes T(1) \cong \bigoplus_{m \in A_{2s+t}} a_m T(m)\otimes T(1).$$
Now by (2.2g) we find that $T(m+1) \mid \Delta_{s,t+1}$ for all $m \in A_{2s+t}$. We note that when $t$ is odd $A_{2s+(t+1)} = \{ m+1: m \in A_{2s+t}\}$, whilst when $t$ is even, $A_{2s+(t+1)} = \{m+1: m \in A_{2s+t}\} \cup \{1\}$. Thus in order to show that the result holds for $\Delta_{s,t+1}$ it remains to show that $T(1) \mid \Delta_{s,t+1}$ whenever $t$ is even. When $t$ is even we have by induction that $T(2) \mid \Delta_{s,t}$ and thus $T(2)\otimes T(1)\mid \Delta_{s,t+1}$. By \eqref{esquare} and the fact that $E \cong T(1)$ we see that $T(2) \otimes T(1) \cong E^{\otimes 3}$. Thus \eqref{ecube} yields that $T(1) \mid \Delta_{s,t+1}$.

From \eqref{star} we also obtain
$$ \Delta_{s+1,t} = \Delta(2) \otimes \Delta_{s,t} \cong \bigoplus_{m \in A_{2s+t}} a_m \Delta(2)\otimes T(m).$$
Since $T(m)$ has Weyl filtration with sections $\Delta(m), \Delta(m_1), \cdots, \Delta(m_k)$ where $m > m_1, \ldots, m_k$ we deduce by (2.2f) that $\Delta(2)\otimes T(m)$ has Weyl filtration with sections $\Delta(m+2), \Delta(n_1), \cdots, \Delta(n_l)$ where $m+2 > n_1, \ldots, n_l$. By Lemma \ref{deltatilt} we have $\Delta_{s+1,t}\mid E^{\otimes 2(s+1)+t}$ and it follows that each of the summands $\Delta(2)\otimes T(m)$ occurring on the right-hand side above must decompose as a direct sum of indecomposable tilting modules. By consideration of highest weights we deduce that $T(m+2) \mid \Delta(2)\otimes T(m)$ and hence $T(m+2) \mid \Delta_{s+1,t}$ for all $m \in A_{2s+t}$. We note that when $t$ is odd $A_{2(s+1)+t} = \{k+2: k \in A_{2s+t}\} \cup \{1\}$, whilst when $t$ is even $A_{2(s+1)+t} = \{ k+2 : k \in A_{2s+t}\} \cup \{2\}$. Thus in order to show that the result holds for $\Delta_{s+1,t}$ it remains to show that $T(1) \mid \Delta_{s+1,t}$ whenever $t$ is odd and $T(2) \mid \Delta_{s+1,t}$ whenever $t$ is even. When $t$ is odd we have by induction that $T(1) \mid \Delta_{s,t}$ giving $\Delta(2) \otimes T(1)\mid \Delta_{s+1,t}$. Hence by \eqref{delta11} we see that $T(1) \mid \Delta_{s,t+1}$. When $t$ is even we have by induction that $T(2) \mid \Delta_{s,t}$ giving $\Delta(2) \otimes T(2)\mid \Delta_{s+1,t}$. By (2.2e) and (2.2f) we find that $\Delta(2) \otimes T(2)$ has Weyl filtration with quotients $\Delta(4), \Delta(2), \Delta(2), \Delta(0)$. Since $\Delta(2)\otimes T(2)$ must decompose as a direct sum of indecomposable tilting modules, consideration of highest weights yields $\Delta(2)\otimes T(2) \cong T(4) \oplus T(2)$ and hence $T(2) \mid \Delta_{s+1,t}$.
\end{proof}

\begin{corollary}
\label{Dparts}
Let $K$ be an infinite field of characteristic $2$, let $G=\GL(2,K)$ and let $E$ denote the
natural $KG$-module. Let $s,t\geq 1$. Then $T(\lambda) \mid D_{s,t}$ if and only if $\lambda=(\lambda_1, \lambda_2)$ with $\lambda_1>\lambda_2 \geq t$.
\end{corollary}

\begin{proof}
By Corollary \ref{Dtilt} we see that $D_{s,t} \mid E^{\otimes 2s +3t}$. Thus $D_{s,t}$ decomposes as a direct sum of
indecomposable tilting modules of the form $T(\lambda)$ where $\lambda=(\lambda_1, \lambda_2)$ is a partition of
$2s+3t$ into at most two parts with $\lambda_1 > \lambda_2$. Let $\lambda$ be such a partition. The restriction of $D_{s,t}$ to $\SL(2,K)$ is denoted by $\Delta_{s,t}$ and thus $T(\lambda)$ is a direct summand of $D_{s,t}$ if and only if $T(\lambda_1-\lambda_2)$ is a direct summand of $\Delta_{s,t}$. By Lemma \ref{Deltaparts}, this happens if and only if $\lambda_1-\lambda_2 \in A_{2s+t}$. That is, if and only if $0 < \lambda_1-\lambda_2 \leq 2s+t$ and $\lambda_1-\lambda_2 \equiv 2s+t \bmod 2$. Since $\lambda$ is a partition of $2s+3t$, we also have $\lambda_1 + \lambda_2 = 2s+3t$. Hence we deduce that $T(\lambda) \mid D_{s,t}$ if and only if $\lambda=(\lambda_1, \lambda_2)$ with $\lambda_1>\lambda_2 \geq t$.
\end{proof}

\begin{proposition}
\label{mostparts}
Let $K$ be an infinite field of characteristic $2$, let $G=\GL(2,K)$ and let $E$ denote the
natural $KG$-module. Let $r$ be a positive integer greater than $6$ and let $\lambda \in \Lambda_2^+(r)$.
\begin{itemize}
\item[(i)] If $r$ is odd and $\lambda \neq (r)$ then $T(\lambda) \mid L^r(E)$.
\item[(ii)] If $r$ is even and $\lambda \neq (r), (r-1,1)$ then $T(\lambda) \mid L^r(E)$.
\end{itemize}
\end{proposition}
\begin{proof}
(i) Write $r=2k+1$ where $k \geq 3$. By equation \eqref{Ralpheq}, $D_{k-1,1} \mid L^r(E)$. Thus it is enough to show that $T(\lambda) \mid D_{k-1,1}$ for all $\lambda \neq (r)$. This follows immediately from Corollary \ref{Dparts}.

(ii) Let $r=2k$ where $k>3$. By equation \eqref{Ralpheq}, $D_{k-3,2} \mid L^r(E)$. Thus it is enough to show that $T(\lambda) \mid D_{k-3,2}$ for all $\lambda \neq (r), (r-1,1)$. This follows from Corollary \ref{Dparts}.
\end{proof}

The direct sum decomposition given in \eqref{Ralpheq} also allows us to find lower bounds for the multiplicities of the indecomposable tilting modules occurring up to isomorphism as direct summands of $L^r(E)$. In fact we shall see that these multiplicities are large in general. Restricting \eqref{Ralpheq} to $\SL(2,K)$ yields that $\Delta_{s,t}$ is a summand of $L^r(E)$ whenever $r=2s+3t$ and $s, t\geq 1$.

\begin{lemma}\label{bound}
Let $K$ be an infinite field of characteristic $2$, let $H=\SL(2,K)$ and let $E$ denote the
natural $KH$-module. Let $r=2s+3t$ with $s, t\geq 1$. Then $E^{\otimes t} \mid \Delta_{s,t}$.
\end{lemma}

\begin{proof} By Lemma \ref{Deltaparts} and the fact that $E \cong T(1)$ we know that $E \mid \Delta_{s,1}$ as $\SL(2,K)$-modules. Hence $E^{\otimes t} \mid \Delta_{s,1}\otimes E^{\otimes t-1} = \Delta_{s,t}$.
\end{proof}

\begin{corollary}\label{bound-cor}
Let $K$ be an infinite field of characteristic $2$, $G=\GL(2,K)$ and let $E$ denote the
natural $KG$-module. Let $r$ be a positive integer and let $\{(s_i,t_i) :i=1, \ldots, k\}$ be a complete set of solutions to the equation $r=2s+3t$.
For each partition $\lambda=(\lambda_1, \lambda_2)$ of $r$ with $0<\lambda_1-\lambda_2 \leq t_i$ define $\lambda(i) = (\lambda_1 - (s_i+t_i), \lambda_2 - (s_i+t_i))$. Then the multiplicity of $T(\lambda)$ in $L^r(E)$ is bounded below by
$$\sum m_{s_i,t_i} d_{\lambda(i)},$$
where the sum ranges over all $i$ such that $0< \lambda_1-\lambda_2 \leq t_i$.
\end{corollary}

\begin{proof}
By \eqref{Ralpheq} we have $\bigoplus_i m_{s_i,t_i}D_{s_i,t_i} \mid L^r(E)$. Thus the multiplicity of $T(\lambda)$ in $L^r(E)$ is greater than or equal to the multiplicity of $T(\lambda)$ in $\bigoplus_i m_{s_i,t_i}D_{s_i,t_i}$. Restriction to $\SL(2,K)$ yields that the multiplicity of $T(\lambda)$ in $L^r(E)$ is greater than or equal to the multiplicity of $T(\lambda_1-\lambda_2)$ in $\bigoplus_i m_{s_i,t_i}\Delta_{s_i,t_i}$. By Lemma \ref{bound} we see that $E^{\otimes t_i} \mid \Delta_{s_i,t_i}$. Thus the multiplicity of $T(\lambda)$ in $L^r(E)$ is greater than or equal to the multiplicity of $T(\lambda_1-\lambda_2)$ in $\bigoplus_i m_{s_i,t_i}E^{\otimes t_i}$ and the result now follows from \eqref{tensortilt} restricted to $\SL(2,K)$.
\end{proof}

Notice that Proposition \ref{mostparts} and Lemma \ref{bound} go most of the way to proving Theorem A. Indeed, it only remains to show that $T(r-1,1) \mid L^r(E)$ if and only if $r$ is not a power of $2$. We shall prove this result in the following section.

In the remainder of this section we shall prove that, for $r>6$, $L^r(E)$ is a tilting module if and only if $r$ is odd.

\begin{lemma}
\label{Lienontilt}
\begin{itemize}
\item[(i)] $L^2(T(2))$ contains a non-tilting summand.
\item[(ii)] $L^2(T(3))$ contains a non-tilting summand.
\end{itemize}
\end{lemma}

\begin{proof}
(i) By \eqref{esquare}, $T(2) \cong E^{\otimes 2}$ and so $T(2)$ has basis $\{e_1,e_2, e_3, e_4\}$, where $e_1$ has weight 2, $e_2$ has weight $-2$ and $e_3, e_4$ have weight zero (for example, in terms of our basis for $E$
we may identify $e_1$ with $x \otimes x$, $e_2$ with $y \otimes y$ and so on).
One checks that  the only weights occurring in $L^2(T(2))$ are $0, \pm 2$ and that the composition factors are $L(2), L(2), L(0), L(0)$. Suppose for contradiction that $L^2(T(2))$ is a tilting module. Then, by consideration of highest weights, we must have that $T(2) \mid L^2(T(2))$. Since $T(2)$ has composition factors $L(2), L(0), L(0)$, this would leave only $L(2)$ in the complement which is non-tilting, thus contradicting the assumption that $L^2(T(2))$ is a tilting module.

(ii) By (2.2b) and (2.2a)  $T(3) \cong \nabla^3(E) \cong S^3(E)$. Thus $T(3)$ has basis $\{e_1,e_2, e_3, e_4\}$, where $e_1$ has weight 3, $e_2$ has weight $-3$, $e_3$ has weight 1 and $e_4$ has weight -1. It is then easy to see that the only weights occurring in $L^2(T(3))$ are $0, \pm 2, \pm 4$ and that the composition factors are $L(4), L(2), L(0), L(0)$. Suppose for contradiction that $L^2(T(3))$ is a tilting module. Then, by consideration of highest weights, we must have that $T(4) \mid L^2(T(3))$. However, $T(4)$ has composition factors $L(4), L(2), L(2), L(0), L(0)$, so it cannot be a direct summand of $L^2(T(3))$.
\end{proof}

\begin{theorem}
\label{Lietilt}
Let $r>6$. Then $L^r(E)$ is a tilting module if and only if $r$ is odd.
\end{theorem}

\begin{proof} If $r$ is odd then $L^r(E)$ is a summand of the tensor power,
as explained in the introduction. To prove the converse, we consider the restriction to $H=\SL(2,K)$.
It suffices to prove that for all even $r$ with $r>6$ we have that $L^r(E)$ contains a non-tilting summand.
We will use the following argument. Suppose $W= U\oplus V$ as $KH$-modules, then $L^2(U)$ is a direct summand of
$L^2(W)$. Namely, it is clear that there is a vector space decomposition
$$L^2(W) = L^2(U)\oplus L^2(V) \oplus [U, V] $$
and it is then easy to check that each summand is a $KH$-module.

Write $r=2k$ and suppose first that $k$ is odd.
Since $r>6$ we have that $k>3$ and hence $L^2(\Delta_{\frac{k-3}{2},1})$ is a direct summand of $L^r(E)$, by \eqref{Ralpheq}. By Lemma \ref{Deltaparts} we see that $T(3)$ is a direct summand of $\Delta_{s,1}$ for all $s \geq 1$. Thus, by
the above argument,  $L^2(T(3))$ is a direct summand of $L^r(E)$ and hence $L^r(E)$ contains a non-tilting summand by Lemma \ref{Lienontilt}~(ii).

Next suppose that $k$ is even. For $k>6$ we have that $L^2(\Delta_{\frac{k}{2}-3,2})$ is a direct summand of $L^r(E)$, by \eqref{Ralpheq}. Then by Lemma \ref{Deltaparts} we see that $T(2)$ is a direct summand of $\Delta_{s,2}$ for all $s \geq 1$. Thus, again by the above argument,  $L^2(T(2))$ is a direct summand of $L^r(E)$ and hence $L^r(E)$ contains a non-tilting summand by Lemma \ref{Lienontilt}~(i).

It remains to deal with the cases $k=4$ and $k=6$. For $k=4$ equation \eqref{Ralpheq} gives
$$L^8(E) \cong L^4(\Delta(2)) \oplus \Delta_{1,2}.$$
By arguments similar to those in Lemma \ref{Lienontilt} it is easy to show that $L^4(\Delta(2))$ has composition factors
$$L(6), L(4), L(4), L(2), L(2), L(2),L(0), L(0), L(0), L(0).$$
Thus if $L^4(\Delta(2))$ is a tilting module, it must contain $T(6)$ as a direct summand. Since $T(6)$ has composition factors
$$L(6), L(4), L(4), L(2), L(2),L(0), L(0), L(0), L(0),$$
this would leave only $L(2)$ in the complement, which is non-tilting. Thus $L^8(E)$ contains a non-tilting summand.
For $k=6$, equation \eqref{Ralpheq} gives that $L^4(E)$ is a direct summand of $L^{12}(E)$ and it is easy to check that $L^4(E)$ has
composition factors $L(2), L(0)$, hence is not tilting,  so that $L^{12}(E)$ contains a non-tilting summand.
\end{proof}

Note that it can be shown by direct computation that $L^2(E) \cong T(1,1)$, $L^4(E) \cong \nabla(3,1)$ and $L^6(E) \cong T(5,1) \oplus T(3,3)$ as $\GL(2,K)$-modules.

\section{Lie powers of the natural module in arbitrary  characteristic}
We now return to the case where $K$ is an infinite field of arbitrary prime characteristic $p$. As before we let $G=\GL(2,K)$, $H=\SL(2,K)$ and let $E$ denote the natural $KG$-module with canonical basis $\{x,y\}$, as described in Section \ref{n=2}.

In this section we shall show that $T(r-1,1) \mid L^r(E)$ if and only if $r$ is not
a power of $p$. To do so,  we will exploit the fact that the highest weight in $L^r(E)$, namely $(r-1,1)$, has one-dimensional weight space.

\begin{remark}\label{4remark} \normalfont \ If $r$ is not divisible by $p$
then, as we have seen in \eqref{Lietilts}, $L^r(E)$ is a  tilting module. Since the
$(r-1,1)$ weight space is one-dimensional it follows immediately in this case that
$L^r(E)$ has a unique summand isomorphic to $T(r-1,1)$.
\end{remark}

Thus for the rest of this section we shall assume that $r$ is divisible by $p$.
The $(r-1,1)$ weight space of $L^r(E)$ is spanned by the left-normed Lie monomial
$$\zeta:= [\ldots [[y,x],x],\ldots,x].
$$
We claim that if $T(r-1,1)$ is a submodule of $L^r(E)$ then the $KG$-submodule of $L^r(E)$ generated by $\zeta$, denoted $G \zeta$,  is isomorphic to $\Delta(r-1,1)$. Indeed, suppose that $T(r-1,1)$ is a submodule of $L^r(E)$. Then $\zeta$ must be
contained in $T(r-1,1)$ and it follows that $G\zeta$ is a submodule of $T(r-1,1)$. By Remark \ref{weight}  this implies that
$G\zeta\cong \Delta(r-1,1)$.

\begin{lemma}\label{weyl}
Let $K$ be an infinite field of prime characteristic $p$ and let $r$ be a positive multiple of $p$. \\
(i) There is a homomorphism $\varphi: T(r-1,1) \to L^r(E)$, which maps the unique submodule of $T(r-1,1)$
isomorphic to $\Delta(r-1,1)$ onto the submodule $G\zeta$ of
$L^r(E)$.\\
(ii) $T(r-1,1)$ is a submodule of $L^r(E)$ if and only if
$G\zeta \cong \Delta(r-1,1)$.
\end{lemma}

\begin{proof}
(i) We have a short exact sequence of $KG$-modules
$$0\to U \to L^{r-1}(E)\otimes E \stackrel{\psi}\to L^r(E)\to 0
$$
(where $\psi$ maps $z_1\otimes z_2$ to $[z_1, z_2]$).
Since $p$ does not divide $r-1$ we may apply Remark \ref{4remark} to find that
$L^{r-1}(E)$ has a unique summand isomorphic to the tilting module
$T(r-2,1)$.
By (2.2g), noting that $E\cong T(1)$,
the module $T(r-2,1)\otimes E$
has $T(r-1,1)$ as a summand. The $(r-1,1)$ weight space of $L^{r-1}(E)\otimes E$
is one-dimensional, spanned by $\zeta'\otimes x$ where
$\zeta':= [\ldots [[y,x],x],\ldots,x]$ spans the
$(r-2,1)$ weight space of $L^{r-1}(E)$.
Hence $\zeta'\otimes x$ must lie in the summand $T(r-1,1)$. By
Remark \ref{weight},  the submodule $G(\zeta'\otimes x)$
of $T(r-1,1)$ is the unique submodule isomorphic to the Weyl module $\Delta(r-1,1)$.
Let $\varphi$ be the restriction of
$\psi$ to $T(r-1,1)$. Then $\varphi(\zeta'\otimes x)=\zeta$ and so $\varphi$ maps $G(\zeta'\otimes x)$ onto $G\zeta$. This completes the proof of part (i).

(ii) By the remark preceding the lemma, we know that if
$T(r-1,1)$ is a submodule of $L^r(E)$ (not necessarily via $\varphi$) then
$G\zeta$ must be isomorphic to $\Delta(r-1,1)$.
Conversely, if $G\zeta \cong \Delta(r-1,1)$ then the restriction
of $\varphi$ to $\Delta(r-1,1)$ is one-to-one. Since the socle of
$T(r-1,1)$ is simple (this follows from Lemmas 5 and 11 in~\cite{EH}),
and is contained in $\Delta(r-1,1)$ it follows that $\varphi$
is one-to-one and hence $T(r-1,1)$ is isomorphic to a submodule of $L^r(E)$.
\end{proof}

In order to determine whether $T(r-1,1) \mid L^r(E)$ we must
study the $KG$-submodule $G\zeta$ of $L^r(E)$ generated by $\zeta$.
By Lemma \ref{weyl}~(i), this is a factor module of $\Delta(r-1,1)$.
In particular its (non-zero) weight spaces are one-dimensional.
Our goal is to determine weight spaces of $G\zeta$ for sufficiently many
weights, so that we can identify its composition factors.
Certainly $L(r-1,1)$ occurs
since $\zeta \in G\zeta$.

Let $g=\left(\begin{matrix}s & u \cr t & v\end{matrix}\right) \in G$, then
$g\zeta = (ux+vy)({\rm ad}(sx+ ty))^{r-1}.$
If $s=0$ then this is just a scalar multiple of $[\ldots [[x,y],y],\ldots,y]$, which is in
a one-dimensional weight space of a weight of $L(r-1,1)$. So we assume now
 $s\neq 0$, and then without loss of generality, $s=1$.
Let $\alpha = x+ty$. Since $(\alpha)({\rm ad}(\alpha))^{r-1}=0$,
the elements
$(x)({\rm ad} \alpha)^{r-1}$ and $(y)({\rm ad}\alpha)^{r-1}$ are
linearly dependent, so to identify weight spaces of
$G\zeta$, it is enough to consider
 the second of these two.

Let $R_{\alpha}$ and $L_{\alpha}$ be right- and left multiplication by $\alpha$ in the associative tensor algebra on $E$. These operations commute and
$$({\rm ad} \alpha)^{r-1} = \sum_{k=0}^{r-1}{r-1\choose k} R_{\alpha}^{r-1-k}(-1)^kL_{\alpha}^k.$$
Hence we have
\begin{equation}
(y)({\rm ad} \alpha)^{r-1}
= \sum_{k=0}^{r-1} (-1)^k{r-1\choose k} (x+ty)^k y (x+ty)^{r-1-k}.
\label{**}
\end{equation}
If  $r=p^m$ we have
 ${p^m-1\choose k}\equiv (-1)^k \bmod p$, and \eqref{**} specializes to
\begin{equation}
(y)({\rm ad} \alpha)^{p^m-1}
= \sum_{k=0}^{p^m-1} (x+ty)^k y (x+ty)^{p^m-1-k}.
\label{***}
\end{equation}
Note that \eqref{**} and \eqref{***} are expressions in the associative tensor algebra of $E$.
If we write either of these  as a polynomial in $t$, then for each $i$ the coefficient
of $t^i$ is a weight vector, with $i+1$ copies of $y$ and $r-(i+1)$
copies of $x$. Hence for different values of $i$ the weights are
distinct. Since the field is infinite, the module $G\zeta$ has a
basis consisting of the coefficients of the $t^i$ which are non-zero.

\begin{lemma}\label{ppower}  Assume $r>1$ is a power of $p$.
Then $G\zeta$ is a simple module isomorphic to $L(r-1,1)$.
Moreover $G\zeta \cong \Delta(r-1,1)$ if and only if $r=p$.
\end{lemma}

\begin{proof} Let $r=p^m$, where $m \geq 1$. The idea is to show that $G\zeta$ is a module of the right dimension. Since $L(r-1,1)$ is a composition factor of $G\zeta$, it is enough to show that $\dim G\zeta \leq \dim L(r-1,1)$.
We use equation \eqref{***}.

(i) \ First we show that for $i=cp-1$, where $1\leq c \leq p^{m-1}$, the coefficient of $t^i$ in \eqref{***} is zero. Let $\eta$ be a monomial in $x$ and $y$ of weight $(p^m-cp, cp)$. We shall find
the coefficient of $\eta$ in \eqref{***}. The monomial $\eta$ has the form
$$\eta = x^{a_0}yx^{a_1}y\ldots yx^{a_{cp}},
$$
where $a_j \geq 0$ for all $j$ and $\sum_j a_j = p^m-cp$. (Note that there are
$cp$ copies of $y$ in total.) Then $\eta$ occurs precisely $cp$ times in \eqref{***}, namely for
the following values of $k$,
$$a_0, \  a_0+a_1+1, \ \ldots, \ a_0+a_1+ \ldots + a_{cp-1} + (cp-1).
$$
Hence the coefficient of $\eta$ in \eqref{***}  is equal to
$$cpt^{cp-1} \equiv 0 \bmod p.
$$
(ii) \ Since $G\zeta$ has basis consisting of the coefficients of the $t^i$ which are non-zero, we compute an upper bound for the dimension of $G\zeta$ as follows. The total number of possible  weight vectors  in $G\zeta$ is $p^m$, and we have shown that $p^{m-1}$ of these are zero. Therefore $\dim G\zeta \leq p^m-p^{m-1}$.
Thus it suffices to show that $\dim L(r-1,1) = p^m-p^{m-1}$.
This is clear from (2.2b) and (2.2a) if $m=1$. For $m\geq 2$ we have
$r-2= p-2 + \sum_{j=1}^{m-1} (p-1)p^j$. Thus, restricting to $SL(2,K)$ yields
$$L(r-2) \cong L(p-2) \otimes
(\bigotimes_{j=0}^{m-2} L(p-1)^{F^j}),
$$
by Steinberg's tensor product theorem.
Applying (2.2b) and (2.2a), we see that the dimension of $L(r-2)$ is $(p-1)p^{m-1}$, as required.
Finally we note that $\dim \Delta(r-1,1)= r-1=p^m-1$ and therefore
$\Delta(r-1,1) \cong L(r-1,1)$ if and only if $m=1$.
\end{proof}

\begin{lemma}\label{no-p-power}  Let $r$ be a positive multiple of $p$ that is not a power of $p$.
Then $G\zeta \cong \Delta(r-1,1)$.
\end{lemma}

\begin{proof}
We show that $G\zeta$ is a module of the right dimension. Since $G\zeta$ is a factor module of $\Delta(r-1,1)$, it is enough to show that $\dim G\zeta \geq \dim \Delta(r-1,1)=r-1$.
We use equation \eqref{**}.

Consider the weight $(r-v,v)$ with $1\leq v\leq r/2$. We will show
that the coefficient of $t^{v-1}$ in \eqref{**}
is non-zero, so that the $(r-v,v)$ weight space is non-zero.
By applying a group element that interchanges $x$ with $y$ (up to a sign)
it will follow that the $(v, r-v)$ weight space is also non-zero.
This will give in total $r-1$ distinct non-zero weight spaces, which will prove the
lemma.

Consider arbitrary
monomials $\eta = yx^{a_1}yx^{a_2}...yx^{a_v}$ of weight $(r-v,v)$ where $v$ is fixed. It is enough to show
that at least one of these monomials occurs with non-zero coefficient in \eqref{**}.

Such $\eta$ occurs in \eqref{**} for
$$k=0, \ a_1+1, \ (a_1+1)+(a_2+1), \ldots \, \  , \sum_{i=1}^{v-1} (a_i+1).
$$
The coefficient of $\eta$ in $(y)({\rm ad} \alpha)^{r-1}$ is therefore equal to
$$t^{v-1}\left(1 + (-1)^{a_1+1}{r-1\choose a_1+1} + \ldots + (-1)^{\sum_{i=1}^{v-1} (a_i + 1)}{r-1\choose \sum_{i=1}^{v-1} (a_i + 1)}\right).
$$
Suppose, for contradiction, that this coefficient is equal to zero modulo $p$ for every such monomial $\eta$. Then taking
$a_1= \ldots = a_{v-1}= 0$ gives
\begin{equation}
1 + \sum_{i=1}^{v-1}(-1)^i {r-1\choose i} \equiv 0
\ (\bmod p).
\label{(1)}
\end{equation}
Now take any  $w$ with $v-1 < w \leq r-1$, and take the monomial
$\eta$ with $a_1=1$, and $a_2= \ldots = a_{v-2}=0$ and
$a_{v-1} = w-v$. (Note that $\sum_{i=1}^{v-1} a_i \leq r-v$, so such monomial
is defined).
Since the coefficient of $\eta$ is equal to zero modulo $p$ we obtain
\begin{equation}
1 + \sum_{i=2}^{v-1}(-1)^i {r-1\choose i} + (-1)^w {r-1\choose w} \equiv 0 \
(\bmod p)
\label{(2)}.
\end{equation}
Subtracting \eqref{(2)} from \eqref{(1)} yields
\begin{equation}
(-1)^w{r-1\choose w} \equiv 1  \ (\bmod p),
\label{(dagger)}
\end{equation}
since $r\equiv 0$ ($\bmod p$).
Note that if we let $w$ vary this tells us that all binomial coefficients
$\binom{r-1}{w}$ are all non-zero modulo $p$, since we are allowed to take
$w$ to be any integer in the range $r/2 \leq w \leq r-1$.

If $r-1$ is odd and $p\neq 2$ then taking $w=r-1$ gives $-1\equiv 1$ ($\bmod p$);
a contradiction. If $r-1$ is odd and $p=2$, then all entries in the $(r-1)$-th row of
Pascal's triangle are equal to $1$ modulo 2 and it follows that
${r\choose k} \equiv 0$ ($\bmod 2$) for $1\leq k\leq r-1$. Thus $r$ must be a power of $2$, contradicting the hypothesis.

Now assume $r-1$ is even. By $\eqref{(dagger)}$,  the $(r-1)$-th row of Pascal's triangle
modulo $p$ has entries $1$ and $(-1)$ alternating, and we once more deduce
that ${r\choose k} \equiv 0$ ($\bmod p$) for $1\leq k \leq r-1$. Thus $r$ must be a power of $p$, contrary to the hypothesis.
\end{proof}

We shall show that whenever $T(r-1,1)$ is a submodule of $L^r(E)$, it is a direct summand. This will use the following:

\begin{proposition}\label{inj} Let $s\geq 0$.
\begin{itemize}
\item[(i)] The tilting module $T(s)$ is injective in degree $s$.
\item[(ii)] If  $s+2$ is not a power of $p$ then $T(s)$ is injective
in degree $s+2$.
\end{itemize}
\end{proposition}

\begin{proof} A tilting module is injective if and only if it is projective
since it is self-dual. We show projectivity.

(i) \ For $0\leq s\leq p-1$ the Schur algebra $S(2,s)$ is semisimple, so
any module is projective.
For $p\leq s\leq 2p-2$,
$T(s)$ is projective, for example by~\cite[Lemma 20, Lemma 24]{EH}
(with $u=0$ and $s=0$ respectively). Now let
$s> 2p-2$ and assume true for all weights $< s$.

Suppose first that $s=kp+(p-1)$. Then $T(s)\cong T(p-1)\otimes T(k)^F$, by (2.2d). By induction,
$T(k)$ is projective in degree $k$. The functor $(-)^F\otimes T(p-1)$ is an equivalence
between the block containing $k$ and the block containing $s$
(see for example Lemma 1 in~\cite{EH}), hence
$T(s)$ is projective in degree $s$.

Now suppose that $s=kp+(p+i)$ where $0\leq i\leq p-2$. Then
$T(s)\cong T(p+i) \otimes T(k)^F$ by (2.2c). By induction, $T(k)$ is projective
in degree $k$, say $T(k) = P(u)$, the projective cover of the simple
module $L(u)$.
By Lemma 11 of~\cite{EH}, the module $T(s)$ has top $L(pu+j)$ where $i+j=p-2$.
Now the arguments in Lemma 20 and Lemma 24 of~\cite{EH}
show that $T(p+i)\otimes T(k)^F$ is projective in
degree $s$.

(ii) \ Assume  $s+2$ is not a power of $p$. By part~(i) we have that $T(s)$ is projective in degree $s$.
Let $T(s) = P_s(u)$, the projective module in degree $s$
with simple top $L(u)$.
In degree $s+2$ there is then a surjective homomorphism
$$ P_{s+2}(u) \stackrel{\pi}\to P_s(u),
$$
where $P_{s+2}(u)$ is the projective cover of $L(u)$ in degree $s+2$, since
$P_s(u)$ has simple top $L(u)$. To show this is an isomorphism, it suffices
to show that both modules have the same $\Delta$-quotients.
Let $[M:\Delta(t)]$ denote the number of quotients isomorphic
to $\Delta(t)$ in a $\Delta$-filtration of $M$  (if $M$ has $\Delta$-filtration,
this is well-defined).
By `BGG reciprocity' and duality, (see for example~\cite{CPS}),
we have that for any $t\leq r$ and $t\equiv r \bmod 2$,
$$[P_r(u) : \Delta(t)] = [\Delta(t):L(u)]
$$
where $[\Delta(t):L(u)]$ is the  multiplicity of $L(u)$ as a composition
factor of $\Delta(t)$.

Since $s \equiv s+2 \bmod 2$, it follows that if $t\leq s$ then $[P_r(u) : \Delta(t)]$ is the same
for $r=s$ and $r=s+2$. It remains to show that $\Delta(s+2)$ is not a $\Delta$-quotient of $P_{s+2}(u)$, or equivalently, that $L(u)$ does not occur as a composition factor of $\Delta(s+2)$. We know that
$L(u)$ is the socle of $T(s)$ and is therefore the socle of
$\Delta(s)$. Since $s+2$ is not a power of $p$ we can deduce
using (2.2h)
that $u\neq 0$.
The claim follows now from the next lemma.
\end{proof}

\begin{lemma} Suppose that $L(w)= {\rm soc} \Delta(k)$ with $w\neq 0$. If
$L(w)$ is a composition factor of $\Delta(k+2)$ then $k+2$ is a power of
$p$.
\end{lemma}

\begin{proof}
Suppose first that $\Delta(k)$ is simple, so that  $\Delta(k) = L(w)$. Then $L(w)$ takes up $k+1$ dimensions from
$\Delta(k+2)$, which only has dimension $k+3$.  It follows that $L(k+2)$ must
be two-dimensional,  which means that $k+2=p^a$ for some $a\geq 1$.

Now suppose that $\Delta(k)$ is not simple. Thus $k \geq p$, by (2.2b). We proceed by induction on $k$.
Since  $k$ and $k+2$ are  in the same block, we get that $k\equiv -2 \bmod p$.
This follows from the Theorem in~\cite{D-block}.
Thus we may write $k=pm+ p-2$ where
$m\geq 1$,  and by (2.2h), $L(w)$ is
a submodule of $\Delta(m-1)^F$. We see from this that $L(w)\cong L(v)^F$
where $w=pv$ with $v\neq 0$.

Then the end terms for the sequence (2.2h) of  $\Delta(k+2)$ are
 $\Delta(m+1)^F$ and $\Delta(m)^F\otimes L(p-2)$. Since $\Delta(pm+p-2)$
is  multiplicity-free,
$L(w)$ does not occur in $\Delta(m)^F\otimes L(p-2)$. So
 it must occur in $\Delta(m+1)^F$.
Now we have
$L(v)= {\rm soc} \Delta(m-1)$, and also
  $L(v)$ occurs in $\Delta(m+1)$. By the inductive hypothesis, $m+1=p^a$
for some $a\geq 1$. Therefore $k+2= p^{a+1}$ and the lemma is proved.
\end{proof}

\begin{thm2}
Let $K$ be an infinite field of characteristic $p$, $G = \GL(2,K)$ and let $E$ denote the natural $KG$ module. Then $T(r-1,1)$ is a summand of $L^r(E)$ if and only if either $r=p$ or $r$ is not a power of $p$.
\end{thm2}
\begin{proof} By Remark \ref{4remark} we have that $T(r-1,1) \mid L^r(E)$ whenever $r$ is not divisible by $p$. Thus we may assume that $r$ is a positive multiple of $p$. When $r=p$ it can also be shown that $T(p-1,1)$ is a summand of $L^p(E)$ using~\cite{BSt}. Indeed, it follows from ~\cite[Corollary 3.2 and Lemma 4.2]{BSt} that $L^p(E)$ has a direct summand isomorphic to $\nabla(p-1,1)$ and restriction to $\SL(2,K)$ yields $T(p-2) \cong \nabla(p-2) \mid L^p(E)$, by (2.2b). By the unique lifting of $T(p-2)$ to a tilting module for $G$ (see Section \ref{n=2}) it then follows that $T(p-1,1)$ is isomorphic to a $KG$-summand of $L^r(E)$. Thus we may assume that $r=pk$ where $k>1$.

Suppose first that $r$ is not a $p$-power. By Lemma \ref{no-p-power} and Lemma \ref{weyl}~(ii) we get that $T(r-1,1)$ is isomorphic to a submodule of $L^r(E)$. By Proposition \ref{inj}, $T(r-1,1)$ is injective in degree $r$ and hence it is a summand of $L^r(E)$.

Next suppose that $T(r-1,1)$ is a summand of $L^r(E)$. By Lemma \ref{weyl}~(ii)we have $G\zeta \cong \Delta(r-1,1)$. Suppose for contradiction that $r=p^m$, then by Lemma \ref{ppower} we have $G\zeta \cong L(r-1,1)$ and hence $m=1$, contradicting our assumption that $r=pk$ where $k>1$.
\end{proof}

We note that ~\cite[Corollary 3.2]{BSt} used in proof of Theorem B concerns the $p$th metabelian Lie power. The $r$th metabelian Lie power of the natural module, denoted $M^r(E)$, is certain a quotient of the $r$th Lie power $L^r(E)$ (for details, see \cite[section 1]{S} or \cite{BSt} for example). It was shown in ~\cite[Corollary 3.2]{BSt} that the $p$th metabelian Lie power occurs as a direct summand of the $p$th Lie power (see also \cite[Section 2]{BKS} for an explicit splitting map $M^p(E) \rightarrow L^p(E)$). One may wonder whether this quotient always occurs as a direct summand. We give a partial answer.

\begin{proposition}
Let $K$ be an infinite field of characteristic $p$, $G = \GL(2,K)$ and let $E$ denote the natural $KG$ module. Let $r>1$ be a positive integer that is not a power of $p$. Then $M^r(E)$ occurs as a direct summand
of $L^r(E)$ if and only if $r=2$ or $r=ap^k+1$ for some $2 \leq a\leq p$ and $k\geq 0$.
\end{proposition}

\begin{proof}
Assume that $r$ is not a power of $p$. Then by Remark \ref{4remark}, Lemma \ref{weyl}~(ii) and Lemma \ref{no-p-power} we see that $L^r(E)$ has a submodule isomorphic to $T(r-1,1)$. Notice that $L(r-1,1)$ occurs as a composition factor of $T(r-1,1)$, namely as the top composition factor of the submodule $\Delta(r-1,1)$ of $T(r-1,1)$. Since the $(r-1,1)$ weight space of $L^r(E)$ is one-dimensional, this is the unique composition factor isomorphic to $L(r-1,1)$ in $L^r(E)$.

It was shown in~\cite[Lemma 4.2]{BSt} that for $r \geq 2$ the $r$th metabelian Lie power $M^r(E)$ is isomorphic to the dual Weyl module $\nabla(r-1,1)$ as a module for the general linear group. Suppose that $\nabla(r-1,1)$ occurs as a direct summand of $L^r(E)$ as a module for $GL(2,K)$, then in particular it is a submodule and hence its socle is a submodule of $L^r(E)$. Since the socle of $\nabla(r-1,1)$ is $L(r-1,1)$ it follows that $\Delta(r-1,1)$ must be simple, and by the results of section 2, so is the restriction to $SL(2,K)$. Thus, by (2.2b), we see that this holds if and only if $r-2 = 0$ or $r-2 = ap^k-1$ for some $2 \leq a\leq p$ and some $k\geq 0$.
\end{proof}
For $r=1$ we have, trivially, that $M^r(E)=L^r(E)$. Consider the case where $r>1$ is a power of $p$. Then the situation is different, and by Lemma \ref{ppower} the simple module $L(r-1,1)$ is in the socle of $L^r(E)$. We have seen that ~\cite[Corollary 3.2]{BSt} gives $M^p(E)$ is a direct summand of $L^p(E)$. We calculate some further examples for $p=2$. If $r=4$ then $\nabla(3,1)$ is a summand of $L^r(E)$; on the other hand $\nabla(7,1)$ is not a summand of $L^8(E)$. In general, for $r=p^m$ with $m \geq 2$ we do not know whether $M^r(E)$ occurs as a direct summand of $L^r(E)$.

\medskip
We conclude this section with the proof of Theorem A.
\begin{thm1}
Let $K$ be an infinite field of characteristic $2$, $G = \GL(2,K)$ and let $E$ denote the natural $KG$-module. Let $r$ be a positive integer greater than $6$ and $\lambda$ a $2$-regular partition of $r$ into at most two parts.
\begin{itemize}
\item[(i)] If $r$ is not a power of $2$ then $T(\lambda) \mid L^r(E)$ if and only if $\lambda \neq (r)$.
\item[(ii)] If $r$ is a power of $2$ then $T(\lambda) \mid L^r(E)$ if and only if  $\lambda \neq (r), (r-1, 1)$.
\item[(iii)] Let $1 \leq t_1 <  t_2 <  \ldots < t_k$ be such that
 $r=2s_i+3t_i$ with $s_i \geq 1$. Then $\bigoplus_{i=1}^k E^{\otimes t_i} \mid L^r(E)$, considered as modules for $\SL(2,K)$.
\end{itemize}
\end{thm1}

\begin{proof}
We first note that the $(r)$-weight space of $L^r(E)$ is zero. Thus, by Remark \ref{weight}, $T(r)$ is not a summand of $L^r(E)$ for any $r>1$. Parts (i) and (ii) now follow from Proposition \ref{mostparts} and Theorem B, whilst part (iii) follows from \eqref{Ralpheq} restricted to $\SL(2,K)$ and Lemma \ref{bound}.
\end{proof}

We note that Theorem A~(i) can also be obtained as a special case of~\cite[Theorem 6.8]{BJ}, which is stated for $\GL(n,K)$ modules. However, the methods used in~\cite{BJ} do not apply when $r=p^m$ or $r=2p^m$. Thus, for $n=2$ and $p=2$, we see that Theorem A~(ii) deals with the cases not covered by~\cite[Theorem 6.8]{BJ}. In the following section we give some partial results for $n=2$ and $p>2$ which are not covered by~\cite[Theorem 6.8]{BJ}. Thus we look at Lie powers of degrees $p^m$ and $2p^m$.

\section{Lie powers of degree $p^m$ or $2p^m$ for $p>2$}
We return temporarily to the case $G=\GL(n,K)$ and let $V$ denote the natural $KG$-module. Let $r>1$. Then for all $a, b \geq 1$ satisfying $r=a+b$ we have that $[L^a(V), L^b(V)]$ is a submodule of $L^r(V)$. If $a$ and $b$ are coprime then it follows from~\cite[Theorem 1]{VLS} that $[L^a(V), L^b(V)]$ is isomorphic to the tensor product $L^a(V)\otimes L^b(V)$ as a module for $\GL(n,K)$. Namely, there is a surjective homomorphism from this tensor product onto $[L^a(V), L^b(V)]$, and by~\cite{VLS} both modules have the same dimension. Hence whenever $T(\lambda)$ is a summand of this tensor product, and in addition is injective in degree $r$, it follows that $T(\lambda)$ is a summand of $L^r(V)$.
Therefore it will be advantageous to take temporarily $n=r$ so that all summands of $V^{\otimes r}$ are injective
(see for example 3.7 in~\cite{E}, $V^{\otimes r} \cong S^{(1^r)}(V)$). The truncation functor, see
~\cite[section 6.5]{G}, maps $V$ to $E$, $V^{\otimes r}$ to $E^{\otimes r}$ and $L^r(V)$ to
$L^r(E)$. Moreover, by~\cite{E} given a partition $\lambda=(\lambda_1,\lambda_2)$ , the truncation functor
takes the tilting module $T(\lambda)$ for $\GL(r,K)$ to the tilting module $T(\lambda)$ for $\GL(2,K)$.

\begin{thm3} Let $K$ be an infinite field of odd characteristic $p$, $G=\GL(2,K)$ and
let $E$ be the natural $KG$ module. Let $r>p$ and let $\lambda$ be a partition of $r$ into at most two parts.\\
(i) If  $r=p^m$ with $p>3$ then
then $T(\lambda) \mid L^r(E)$ if and only if $\lambda \neq (r), (r-1,1)$.\\
(ii) Let $r=p^m$ with $p=3$ and suppose $\lambda \neq (r), (r-1, 1),
 ((r+1)/2, (r-1)/2)$. Then $T(\lambda) \mid L^r(E)$.\\
(iii) Let $r=2p^m$ with $p>3$ and suppose $\lambda  \neq (r), (p^m, p^m)$. Then $T(\lambda) \mid L^r(E)$.\\
(iv) Let $r=2p^m$ with $p=3$ and suppose $\lambda \neq (r), (p^m, p^m), (p^m+1, p^m-1), (p^m+2, p^m-2)$. Then $T(\lambda) \mid L^r(E)$.
\end{thm3}

\begin{proof} Note that $r >p > 2$. Let
$V$ be the natural $GL(r,K)$-module, which truncates to $E$ (see the remark above).

\ (i) and (ii) \ Assume first that $r=p^m$.
Recall that the $(r)$-weight space of $L^r(E)$ is zero and hence $T(r)$ is not a summand of $L^r(E)$.
By Theorem B we also have that $T(r-1,1)$ does not occur.
Thus we must show that $T(\lambda)$ occurs when $\lambda_2\geq 2$, with the
exception of $\lambda = ((r+1)/2, (r-1)/2)$ when $p=3$.

Consider the submodule $[L^2(V),L^{r-2}(V)]$ of $L^r(V)$.
By the remarks preceding the theorem we have that $[L^2(V),L^{r-2}(V)] \cong L^2(V)\otimes L^{r-2}(V)$.
We note that since $p \nmid 2$ and $p \nmid r-2$, we may apply~\cite[Corollary 6.10]{BJ} to find $T(1,1) \mid L^2(E)$ and $T(\mu) \mid L^{r-2}(V)$ for all partitions $\mu=(\mu_1,\mu_2)$ of $r-2$ with $\mu_2>0$, except for
$\mu = ((r-1)/2, (r-3)/2))$ when $p=3$. Thus each such $T(1,1) \otimes T(\mu)$ is a submodule of $L^r(V)$ with highest weight $\lambda = (\mu_1+1, \mu_2+1)$. Since $T(1,1) \otimes T(\mu)$ it is a direct sum of tilting modules, it has $T(\lambda)$ as a direct summand. In fact, since $\lambda$ is a $p$-regular partition of $r$, we also have $T(\lambda) \mid V^{\otimes r}$, and since all summands of $V^{\otimes r}$ are injective in degree $r$ it follows that each such  $T(\lambda)$ occurs as a direct summand of $L^r(V)$. Applying the truncation functor now yields that $T(\lambda)$ occurs as a direct summand of $L^r(E)$ for all $\lambda=(\lambda_1, \lambda_2)$ with $\lambda_2 \geq 2$ except for $\lambda = ((r+1)/2, (r-1)/2))$ when $p=3$, as required.

\ (iii) \ Assume next that $r=2p^m$ and $p>3$.
Consider the submodule $[L^3(V),L^{r-3}(V)]$ of $L^r(V)$.
By the remarks preceding the theorem we have that $[L^3(V),L^{r-3}(V)] \cong L^3(V)\otimes L^{r-3}(V)$.
We note that since $p \nmid 3$ and $p \nmid r-3$, we may apply~\cite[Corollary 6.10]{BJ} to find $T(2,1)\mid L^3(V)$
and $T(\mu) \mid L^{r-3}(V)$ for all partitions $\mu=(\mu_1,\mu_2)$ of $r-3$ with $\mu_2>0$.
So each such  $T(2,1) \otimes T(\mu)$ is a submodule of $L^r(V)$ with highest weight $(\mu_1+2, \mu_2+1)$, giving that $T(\lambda)\mid T(2,1) \otimes T(\mu)$ for all partitions $\lambda=(\lambda_1,\lambda_2)$ of $r$ with $\lambda_2 \geq 2$ and $\lambda_1-\lambda_2 >0$. Since each such $\lambda$ is a $p$-regular partition of $r$ we again conclude that $T(\lambda)$ is injective in degree $r$ and hence $T(\lambda) \mid L^r(V)$ for all $\lambda\neq (r), (r-1,1), (p^m, p^m)$. Applying the truncation functor, yields $T(\lambda) \mid L^r(E)$ for all $\lambda\neq (r), (r-1,1), (p^m, p^m)$ and by Theorem B we also know that $T(r-1,1)\mid L^r(E)$.

\ (iv) \ Finally let $r=2p^m$ with $p=3$.
Consider the submodule $[L^5(V),L^{r-5}(V)]$ of $L^r(V)$.
By the remarks preceding the theorem we have that $[L^5(V),L^{r-5}(V)] \cong L^5(V)\otimes L^{r-5}(V)$.
We note that since $p \nmid 5$ and $p \nmid r-5$, we may apply~\cite[Corollary 6.10]{BJ} to find $T(4,1) \mid L^5(V)$
and $T(\mu) \mid L^{r-5}(V)$ for all partitions $\mu=(\mu_1,\mu_2)$ of $r-5$ with $\mu_2>0$ and $\mu_1-\mu_2>1$. So each such $T(4,1) \otimes T(\mu)$ is a submodule of $L^r(V)$ with highest weight $(\mu_1+4, \mu_2+1)$, giving
$T(\lambda) \mid T(4,1) \otimes T(\mu)$ for all partitions $\lambda$ of $r$ with $\lambda_2 \geq 2$ and $\lambda_1-\lambda_2>4$. Each such $T(\lambda)$ is again injective in degree $r$ giving $T(\lambda) \mid L^r(V)$ for all $\lambda\neq (r), (r-1,1), (p^m, p^m), (p^m+1, p^m-1), (p^m+2,p^m-2)$. Applying the truncation functor, yields $T(\lambda) \mid L^r(E)$ for all $\lambda\neq (r), (r-1,1), (p^m, p^m), (p^m+1, p^m-1), (p^m+2,p^m-2)$ and by Theorem B we also know that $T(r-1,1)\mid L^r(E)$.
\end{proof}

\end{document}